\documentclass[12pt]{article}

\usepackage{amsmath}
\usepackage{amsfonts}
\usepackage{latexsym}
\usepackage{amssymb}
\usepackage{stmaryrd}
\usepackage{overpic}
\usepackage{graphicx}

\newtheorem{theorem}{Theorem}[section]

\newtheorem{corollary}[theorem]{Corollary}

\newtheorem{lemma}[theorem]{Lemma}

\newenvironment{proof}[1][Proof]{\textbf{#1.} }
{\ \rule{0.75em}{0.75em}\smallskip}

\begin{document}
\begin{center}
\Large\bf Reliable and efficient a posteriori error estimates of DG methods for
a frictional contact problem\end{center}

\begin{center}
Fei Wang\footnote{Department of Mathematics, Pennsylvania State University, University Park, PA 16802, USA. School of Mathematics and Statistics, Huazhong University of Science and Technology, Wuhan 430074, China.
Email: wangfeitwo@163.com} \quad and \quad
Weimin Han\footnote{Department of Mathematics \& Program in
Applied Mathematical and Computational Sciences, University of
Iowa, Iowa City, Iowa 52242, USA.}
\end{center}

\bigskip
\begin{quote}
{\bf Abstract.} A posteriori error estimators are studied for
discontinuous Galerkin methods for solving a frictional contact
problem, which is a representative elliptic variational inequality
of the second kind. The estimators are derived by relating the error
of the variational inequality to that of a linear problem. 
Reliability and efficiency of the estimators are shown.


\end{quote}

\section{Introduction}

For more than three decades, adaptive finite element method (AFEM) has been
an active research field in scientific computing. As an efficient numerical approach, 
it has been widely used for solving a variety of differential equations. 
Each loop of AFEM consists of four steps,
\[{\rm Solve} \rightarrow {\rm Estimate} \rightarrow {\rm Mark} \rightarrow {\rm Refine}.\]
That is, in each loop, we first solve the problem on an mesh, then use
a posteriori error estimators to mark those elements to be refined, and finally,
refine the marked elements and get a new mesh. We can continue this process until 
the error satisfies certain smallness criterion.  The adaptive finite element method 
can achieve high accuracy with lower memory usage and less computation time.

A posteriori error estimators are computable quantities that
indicate the contribution of error on each element to the global
error. They are used in adaptive algorithms to indicate which elements
need to be refined or coarsened. To capture the true error as precisely as possible, 
they should have two properties: reliability and efficiency
(\cite{ainsworth00,babuska01}). Hence, obtaining reliable and
efficient error estimators is the key for successful adaptive algorithms. A
variety of different a posteriori error estimators have been
proposed and analyzed. Many error estimators can be classified as
residual type or recovery type (\cite{ainsworth00,babuska01}).
Various residual quantities are used to capture lost information
going from $u$ to $u_h$, such as residual of the equation,
residual from derivative discontinuity and so on. Another type of
error estimators is gradient recovery, i.e., $||G(\nabla u_h)-\nabla u_h||$ is used to
approximate $||\nabla u-\nabla u_h||$, where a recovery operator
$G$ is applied to the numerical solution $u_h$ to rebuild the
gradient of the true solution $u$. A posteriori error analysis
have been well established for
standard finite element methods for solving linear partial
differential equations, and we refer the reader to
\cite{ainsworth00,babuska01,verfurth96}.

Due to the inequality feature, it is more difficult to develop a
posteriori error estimators for variational inequalities (VIs).
However, numerous articles can be found on a posteriori error
analysis of finite element methods for the obstacle problem, which is
an elliptic variational inequality (EVI) of the first kind, e.g.,
\cite{bartels04,chen00,kornhuber96,nochetto03,veeser01,yan01}. In
\cite{braess05}, Braess demonstrated that a posteriori error
estimators for finite element solutions of the obstacle problem
can be derived by applying a posteriori error estimates
for an associated linear elliptic problem. For VIs of the second
kind, in \cite{bostan1,bostan2,bostan3,bostan4}, the authors
studied a posteriori error estimation and established a framework
through the duality theory, but the efficiency was not completely
proved. In \cite{wang12}, the ideas in \cite{braess05} were extended 
to give a posteriori error analysis for VIs of the second kind. Moreover, 
a proof was provided for the efficiency of the error estimators.

In recent years, thanks to the flexibility in constructing
feasible local shape function spaces and the advantage to capture
non-smooth or oscillatory solutions effectively, discontinuous Galerkin (DG) methods 
have been widely used for solving various types of partial differential
equations. When applying $h$-adaptive algorithm with standard finite
element methods, one needs to choose the mesh refinement rule
carefully to maintain mesh conformity and shape regularity. In
particular, hanging nodes are not allowed without special treatment. For discontinuous
Galerkin methods, the approximate functions are allowed to be
discontinuous across the element boundaries, so general meshes
with hanging nodes and elements of different shapes are accepted.
Advantages of DG methods include the flexibility of
mesh-refinements and construction of local shape function spaces
($hp$-adaptivity), and the increase of locality in discretization,
which is of particular interest for parallel computing. A
historical account of DG methods' development can be found in
\cite{cockburn00}. In \cite{arnold00,arnold02}, Arnold et al.\ established a unified 
error analysis of nine DG methods for elliptic problems and several articles provided 
a posteriori error analysis of DG methods for elliptic problems 
(e.g.\ \cite{becker03,cai11,castillo05,karakashian03,lazarov08,riviere03}).
Carstensen et al.\ presented a unified approach to a posteriori
error analysis for DG methods in \cite{carstensen09}. In
\cite{wang10}, the authors extended ideas of the unified framework about DG methods for
elliptic problems presented in \cite{arnold02} to
solve the obstacle problem and a simplified frictional contact problem,
and obtained a priori error estimates, which reach
optimal order for linear elements. In \cite{wang11}, reliable a
posteriori error estimators of the residual type were derived for
DG methods for solving the obstacle problem, and efficiency of the
estimators is theoretically explored and numerically confirmed. 
A posteriori error analysis of DG methods for the obstacle problem was also studied 
in \cite{gudi13}. 

In this paper, we study a posteriori error estimates of DG methods
for solving a frictional contact problem. 
The paper is organized as follows: in Section 2 we introduce a frictional contact problem and the DG
schemes for solving it. Then we derive a reliable residual type a posteriori error estimators for the DG methods of a frictional contact problem in Section 3. Finally, we prove efficiency of the proposed error estimators in Section 4. 

\section{A frictional contact problem and DG formulations}
\setcounter{equation}0

\subsection{A frictional contact problem}.

Let $\Omega\subset\mathbb{R}^d$ ($d = 2, 3$) be an open bounded domain with
Lipschitz boundary $\Gamma$ that is divided
into two mutually disjoint parts, i.e., $\Gamma=\Gamma_1\cup \Gamma_2$. 
Here $\Gamma_1$ is a relatively closed subset
of $\Gamma$, and $\Gamma_2=\Gamma\backslash \Gamma_1$. Given
$f\in L^2(\Omega)$ and a constant $g>0$, the frictional contact problem
is: find $u\in V = H^1_{\Gamma_1}(\Omega):=\{v\in H^1(\Omega): v=0 \ {\rm
a.e.\ on}\ \Gamma_1 \}$ such that
\begin{equation}\label{vi}
a(u,v-u) + j(v) - j(u) \geq (f,v-u)\quad\forall\,v\in V,
\end{equation}
where $(\cdot,\cdot)$ denotes the $L^2$ inner product in
the domain $\Omega$ and
\begin{align*}
a(u,v)&=\int_\Omega \nabla u\cdot\nabla v\, dx + \int_\Omega u\,  v\, dx,\\
j(v) &= \int_{\Gamma_2} g\, |v|\, ds.
\end{align*}

The frictional contact problem is an example of elliptic
variational inequalities of the second kind and has a unique
solution $u\in V$ (\cite{duvaut76,glowinski84}).
Moreover, there exists a unique Lagrange multiplier $\lambda\in
L^\infty(\Gamma_2)$ such that
\begin{align}
a(u,v) + \int_{\Gamma_2} g\, \lambda\, v\, ds = (f,v)\quad\forall\,v\in V,\label{lvi}\\
|\lambda|\leq 1, \quad \lambda\, u =|u| \quad {\rm a.e.\ on}\
\Gamma_2.\label{lag_1}
\end{align}
From \eqref{lvi} and \eqref{lag_1}, we know that the solution
$u$ of \eqref{vi} is the weak solution of the following boundary value
problem
\begin{align*}\label{pde1}
 -\triangle u +u &= f  \;\quad\quad {\rm in}\ \Omega,\\
 u &=0 \;\quad\quad {\rm on}\ \Gamma_1,\\
\nabla u \cdot \boldsymbol{n} &= -g\lambda \quad  {\rm on}\ \Gamma_2,
\end{align*}
where $\boldsymbol{n}$ is the unit outward normal vector. For any $v\in V$, set
\begin{equation*}
    \ell(v) = \int_\Omega f\, v\, dx - \int_{\Gamma_2} g\, \lambda\, v\, ds.
\end{equation*}
Then we have by \eqref{lvi}
\begin{equation}\label{eq1}
    a(u,v)=\ell(v) \quad \forall\, v\in V.
\end{equation}

Similar with the argument in \cite{wang12}, given a triangulation $\mathcal{T}_h$ of $\Omega$, for a Lipschitz subdomain $\omega\subset\Omega$, define
\[ a_{\omega,h}(v,w):=\sum_{K\in \mathcal{T}_h}\int_{\omega\cap K} (\nabla v\cdot \nabla w + vw)\, dx \]
and
\[ \|v\|_{1,\omega,h}:= a_{\omega,h}(v,v)^{1/2}. \]
Then define
\begin{equation}
 |\lambda|_{*,\gamma,h}:=\sup\left\{\int_{\gamma} g\, \lambda\,v\, ds : \; v\in
H^1_h(\omega), \; \|v\|_{1,\omega,h}=1\right\}, \label{normg}
\end{equation}
where $\gamma\subset \partial\omega\cap\Gamma_2$ is a measurable subset and 
$H^1_h(\omega) = \{v\in L^2(\omega): v|_{K\cap\omega}\in H^1(K\cap\omega)\}$. 
If $\omega=\Omega$ and $\gamma=\Gamma_2$, the subscript $\omega$ and $\gamma$ are omitted.
We have
\begin{equation}\label{dual_norm}
|\lambda|_{*,\gamma,h}=\|w\|_{1,\omega,h},
\end{equation}
where $w\in H^1_h(\omega)$ is the solution of the following auxiliary equation
\begin{equation}\label{dual_norm1}
a_{\omega,h}(w,v)=\int_{\gamma} g\ \lambda\ v\, ds \quad \forall\,
v\in H^1_h(\omega).
\end{equation}
The formula \eqref{dual_norm} can be proved by an argument similar to that found in \cite{wang12}.

\subsection{Discontinuous Galerkin formulations}

First, we introduce some notations. Let $\{\mathcal{T}_h\}$ be a family of
triangulations of $\overline{\Omega}$ such
that the minimal angle condition is satisfied. For a triangulation
$\mathcal {T}_h$, let $\mathcal{E}_h$ be the set of all
edges, $\mathcal{E}^i_h \subset \mathcal{E}_h$ the set of all interior edges, 
$\mathcal{E}^b_h := \mathcal{E}_h\backslash
\mathcal{E}^i_h$ the set of all boundary edges,
$\mathcal{E}^0_h \subset \mathcal{E}_h$ the set of all edges not lying on $\Gamma_2$, 
$\mathcal{E}^1_h := \mathcal{E}_h^0 \backslash
\mathcal{E}^i_h$, $\mathcal{E}^2_h := \mathcal{E}_h \backslash
\mathcal{E}_h^0$, and define ${\cal E}(K)$ as the set
of sides of $K$. Let $h_K ={\rm diam}(K)$ for $K\in\mathcal
{T}_h$, $h_e = {\rm length}(e)$ for $e\in\mathcal{E}_h$, and
$\mathcal{N}_h$ denote the set of nodes of $\mathcal{T}_h$. For
any element $K\in{\cal T}_h$, define the patch set $\omega_K
:=\cup \{T \in \mathcal{T}_h, \, T \cap K \neq \O \}$, and for any
edge $e$ shared by two elements $K^+$ and $K^-$, define
$\omega_e:=K^+\cup K^-$. For a scalar-valued function $v$ and a
vector-valued function $\boldsymbol{q}$, let $v^i$ = $v|_{\partial K^i}$, $\boldsymbol{q}^i$
= $\boldsymbol{q}|_{\partial K^i}$, and $\boldsymbol{n}^i = \boldsymbol{n}|_{\partial K^i}$ be the unit
normal vector external to $\partial K^i$ with $i=\pm$. Define
the average $\{\cdot\}$ and the jump $[\cdot]$ on an interior edge $e\in{\cal
E}_h^i$ as follows:
\begin{align*}
\{v\} &= \frac{1}{2}(v^+ + v^-),\quad  [v] = v^+\boldsymbol{n}^+
   + v^-\boldsymbol{n}^-,\\
\{\boldsymbol{q}\} &= \frac{1}{2}(\boldsymbol{q}^+ + \boldsymbol{q}^-),\quad
 [\boldsymbol{q}] = \boldsymbol{q}^+\cdot \boldsymbol{n}^+ + \boldsymbol{q}^-\cdot \boldsymbol{n}^-.
\end{align*}
For a boundary edge $e \in \mathcal{E}^b_h$, we let
\[ [v] = v\boldsymbol{n},\quad\{q\}=q,\]
where $\boldsymbol{n}$ is the outward unit normal.

Let us define the following linear finite element spaces
\begin{align*}
 V_h&=\{ v_h \in L^2(\Omega):\;v_h |_K \in P_1(K)\ \forall\,K\in\mathcal{T}_h\},\\
 W_h&=\{ \boldsymbol{w}_h\in[L^2(\Omega)]^2:\; \boldsymbol{w}_h |_K\in [P_1(K)]^2\
 \forall\,K\in\mathcal{T}_h\}.
\end{align*}
We denote by $\nabla_h$ the broken gradient whose restriction on each
element $K \in \mathcal{T}_h$ is equal to $\nabla$. Define some
seminorms and norms by the following relations:
\[\| v\|_{K}^2 =\int_K v^2  dx, \quad |v|^2_{1,K}=\|\nabla
v\|_{K}^2, \quad  \|v\|_{e}^2 =\int_e v^2 ds,\]
\[\|v\|^2_{0,h} = \sum_{K\in \mathcal{T}_h} \|v\|^2_{K}, 
\quad |v|^2_{1,h} = \sum_{K\in \mathcal{T}_h} |v|^2_{1,K},
\quad \|v\|^2_{1,h} = \|v\|^2_{0,h} + |v|^2_{1,h}.\]

Throughout this paper, ``$\lesssim\cdots $" stands for ``$\leq
C\cdots $", where $C$ denotes a generic positive constant
dependent on the minimal angle condition but not on the element
sizes, which may take different values at different occurrences.

Now, let us introduce the Discontinuous Galerkin methods for solving the variational inequality \eqref{vi}.
Here, we take the local DG method (LDG) as an example to show how to derive a
posteriori error estimators of DG methods for solving the frictional
contact problem (\ref{vi}). The derivation and analysis for the LDG method in this paper can be extended to
other DG methods studied in \cite{wang10}. The LDG method
(\cite{cockburn98}) for solving the frictional contact problem is to find $u_h \in V_h$ such that
\begin{equation}\label{dvi}
B_h(u_h,v_h-u_h) + j(v_h) - j(u_h) \geq (f,v_h-u_h)\quad\forall\,
v_h \in V_h,
\end{equation}
where
\begin{align}
B_h(u,v)&=\int_\Omega \left(\nabla_h u\cdot\nabla_h v + u \,v\right) dx
  -\int_{{\cal E}_h^0}[u] \cdot \{\nabla_h v\}\, ds
  -\int_{{\cal E}_h^0} \{\nabla_h u\}\cdot[v]\, ds \nonumber \\
 &{}\quad - \int_{{\cal E}_h^i} \beta \cdot[u] [\nabla_h v]\, ds
   - \int_{{\cal E}_h^i} [\nabla_h u] \beta\cdot [v]\, ds \nonumber \\
   &{}\quad+  \left(r_0([u]) + l(\beta\cdot[u]),r_0([v])
   + l(\beta\cdot [v])\right) + \alpha^j_0(u,v). \label{ldg}
\end{align}
Here $\beta \in [L^2({\cal E}_h^i)]^2$
is a vector-valued function which is constant on each edge of $\mathcal{E}^i_h$, and
$\alpha^j_0(u,v) = \int_{{\cal E}_h^0} \eta [u]\cdot [v]\, ds$ is
the penalty term with the penalty weighting
function $\eta : {\cal E}_h^0 \rightarrow \mathbb{R}$ given by
$\eta_e h_e^{-1}$ on each $e \in {\cal E}_h^0$, $\eta_e$ being a
positive number on $e$. For any $\boldsymbol{w}_h \in W_h$, the lifting operators $r_0 : [L^2({\cal E}_h^0)]^2 \rightarrow
W_h$ and $l : L^2({\cal E}_h^i) \rightarrow W_h$ are defined by
\begin{equation}\label{liftop} 
\int_\Omega r_0(\boldsymbol{q})\cdot \boldsymbol{w}_h dx 
= -\int_{{\cal E}_h^0} \boldsymbol{q}\cdot \{\boldsymbol{w}_h\}\, ds, \quad \int_\Omega
l(v)\cdot \boldsymbol{w}_h  dx = -\int_{{\cal E}_h^i} v\, [\boldsymbol{w}_h]\, ds
\quad\forall\,\boldsymbol{w}_h\in W_h.
\end{equation}

The bilinear form $B_h$ is continuous and elliptic with respect to certain DG-norm, and
therefore, in particular, the discrete problem has a
unique solution $u_h\in V_h$ (see \cite{arnold02, wang10}).
Similar to the continuous problem, there exists a unique Lagrange
multiplier $\lambda_h\in L^\infty(\Gamma_2)$ such that
(\cite{glowinski84})
\begin{align}
B_h(u_h,v_h) + \int_{\Gamma_2} g\, \lambda_h v_h ds = (f,v_h) 
\quad\forall\,v_h\in V_h,
\label{ldvi}\\
|\lambda_h|\leq 1, \quad \lambda_h u_h =|u_h| \quad {\rm a.e.\ on}\ \Gamma_2.
\label{lag_2}
\end{align}
For any $v_h\in V_h$, let
\begin{equation*}
    \ell_h(v_h) = (f,v_h) - \int_{\Gamma_2} g\, \lambda_h v_h ds.
\end{equation*}
Then \eqref{ldvi} becomes
\begin{equation}\label{eq2}
    B_h(u_h,v_h)=\ell_h(v_h) \quad \forall\; v_h\in V_h.
\end{equation}
For any $v\in V$, we know that $[u]=0$ and $[v]=0$ on $e\in {\cal E}_h^0$.  
Then we have from \eqref{lvi} that
\begin{align}
    B_h(u,v) = a (u,v) = \ell(v) \quad  \forall\, v \in V \label{dg_vi}
\end{align}
Obviously, $u_h$ is also the finite element approximation of the
solution $z\in V$ of the linear problem:
\begin{equation}\label{dg_vi2}
B_h(z,v) = \ell_h(v) \quad \forall\, v \in V,
\end{equation}
which is the weak formulation of the boundary value problem
\begin{align}
 -\Delta z + z &= f  \quad\quad\quad {\rm in}\ \Omega,\label{pde2}\\
z &=0 \quad\quad\quad {\rm on}\  \Gamma_1,\nonumber\\
\frac{\partial z}{\partial n} &= -g\lambda_h  \quad\  {\rm on}\
\Gamma_2.\nonumber
\end{align}

\subsection{A bridge between $u_h-u$ and $u_h-z$}

Next, we relate the error $e:=u_h-u$ to $u_h-z$, namely,
\begin{align}
\|e\|_{1,h} + |\lambda-\lambda_h|_{*,h}  \lesssim \|u_h-z\|_{1,h} +
\left(\sum_{e\in{\cal E}_h^0} h_e^{-1}\|[u_h]\|_e^2\right)^{1/2}. \label{post_obst}
\end{align}
Then we will use this relation to derive a posteriori error estimators for DG solutions 
of the frictional contact problem by utilizing a posteriori error estimators of the 
related linear elliptic problem \eqref{pde2}. Note that a similar approach can be 
applied to other elliptic variational inequalities of the second kind.

To derive the inequality \eqref{post_obst}, we first define a continuous piecewise linear 
function in $V_h \cap H^1_{\Gamma_1}(\Omega)$, whose value is close to the numerical solution. 
For any given $v_h \in V_h$, written $v_h=\sum_{K\in \mathcal{T}_h}
\sum_{j=1}^3 \alpha_K^{(j)}\phi_K^{(j)}$, where $\phi_K^{(j)}$, $1\le j\le 3$,
are the linear basis functions corresponding to the three vertices of $K$,
we construct a function $\chi \in V_h \cap H^1_{\Gamma_1}(\Omega)$ as follows: 
At every interior node and the nodes on $\Gamma_2$ of the conforming mesh
$\mathcal{T}_h$, the value of $\chi$ is set to be the average of
the values of $v_h$ computed from all the elements sharing that
node, and $\chi=0$ at the boundary nodes on $\Gamma_1$. For each
$\nu\in \mathcal{N}_h$, let $\omega_\nu=\{K\in
\mathcal{T}_h:\,\nu\in K\}$ and denote its cardinality by
$|\omega_\nu|$, which is bounded by a constant depending only on
the minimal angle condition of the mesh. To each node $\nu$, the
associated basis function $\phi^{(\nu)}$ is given by
\[{\rm supp}\phi^{(\nu)} = \bigcup_{K\in \omega_\nu}K,
\quad \phi^{(\nu)}|_K = \phi_K^{(j)} \ {\rm for}\ x_K^{(j)} =\nu.\] 
Then we define $\chi \in V_h \cap H^1_{\Gamma_1}(\Omega)$ by
\begin{align}\label{chi}
\chi=\sum_{\nu\in \mathcal{N}_h} \beta^{(\nu)}\phi^{(\nu)}, \quad {\rm where} \;
\beta^{(\nu)}=\frac{1}{|\omega_\nu|}\sum_{x_K^{(j)}=\nu}\alpha_K^{(j)}\quad
{\rm if}\ \nu\in\mathcal{N}_h\ {\rm and\ }\nu\not\in\Gamma_1.
\end{align}
For nonconforming meshes, let $\mathcal{N}^0_h$ be the set of all
hanging nodes. Then we construct $\chi$ from $v_h$ same as conforming mesh case 
on all the nodes $\nu\in\mathcal{N}_h\backslash\mathcal{N}^0_h$.
For an upper bound of the error $v_h - \chi$, we quote a result from
\cite{karakashian03} (which is Theorem 2.2 there for conforming meshes;
the same result also holds for nonconforming meshes, which is Theorem 2.3 in \cite{karakashian03}).

\begin{lemma}\label{lem:conform}
Let $\mathcal{T}_h$ be a conforming triangulation.
Then for any $v_h \in V_h$, we can construct a continuous function $\chi \in V_h \cap
H^1_{\Gamma_1}(\Omega)$ from $v_h$, such that
\begin{equation}\label{conform}
\sum_{K\in \mathcal{T}_h}\|v_h-\chi\|_{i,K}^2 \leq C \sum_{e\in
{\cal E}_h^0} h_e^{1-2i}\|[v_h]\|_e^2,\quad i=0,1,
\end{equation}
where the constant $C$ is independent of mesh size and $v_h$ but
which may depend on the lower bound of the minimal angle of the elements in ${\cal T}_h$.
\end{lemma}

Now, let us derive the inequality \eqref{post_obst}. From (\ref{dg_vi}) and (\ref{dg_vi2}), 
for all $v\in V$, we have
\begin{align*}
B_h(u_h-u,v) & = B_h(u_h-z,v) + B_h(z-u,v) = B_h(u_h-z,v) +
\int_{\Gamma_2} g(\lambda-\lambda_h)v\,ds.
\end{align*}
By the definition (\ref{ldg}) and noticing $[v]=0$ on each $e \in {\cal E}_h^0$, the above
equation becomes
\begin{align*}
& \widetilde{a}(e,v) - \int_{{\cal E}_h^0} [e]\cdot\{\nabla_h v\}\, ds 
 - \int_{{\cal E}_h^i}\beta \cdot[e] [\nabla_h v]\, ds\\
&\qquad =\widetilde{a}(u_h-z, v) -\int_{{\cal E}_h^0} [u_h-z]\cdot\{\nabla_h v\}\, ds\\ 
&{}\qquad \quad - \int_{{\cal E}_h^i}\beta \cdot[u_h-z] [\nabla_h v]\, ds + \int_{\Gamma_2}
g(\lambda-\lambda_h)v\,ds,
\end{align*}
where 
\[\widetilde{a}(u,v) = \int_\Omega \left(\nabla_h u\cdot\nabla_h v+ u\,v\right) dx.\]
Then,
\begin{align*}
\widetilde{a}(e, v) &= \widetilde{a}(u_h-z, v)
 - \int_{{\cal E}_h^0} [u-z]\cdot \{\nabla_h v\} ds
 - \int_{{\cal E}_h^i} \beta\cdot[u-z] [\nabla_h v]\, ds
+ \int_{\Gamma_2} g(\lambda-\lambda_h)v\,ds.
\end{align*}
Note that $[u-z]=0$ on each $e \in {\cal E}_h^0$. We have
\begin{equation}
\widetilde{a}(e, v)= \widetilde{a}(u_h-z, v) + \int_{\Gamma_2}
g(\lambda-\lambda_h)v\,ds. \label{equ1}
\end{equation}
Let $\chi\in V_h \cap H^1_{\Gamma_1}(\Omega)$ be the function
constructed from $u_h$, satisfying (\ref{conform}) for $v_h=u_h$.
Taking $v:=\chi-u=\chi-u_h+u_h-u$ in (\ref{equ1}) and using
Cauchy-Schwarz inequality, we have
\begin{align}
\|e\|^2_{1,h}  & \le \|u_h-z\|_{1,h} \left(\|\chi-u_h\|_{1,h} +\|e\|_{1,h}\right)
+ \|e\|_{1,h} \|\chi-u_h\|_{1,h} + \int_{\Gamma_2}g(\lambda-\lambda_h)(\chi-u)\,ds
\nonumber\\
& =\|e\|_{1,h}\left(\|u_h-z\|_{1,h}+\|\chi-u_h\|_{1,h}\right) + \|u_h-z\|_{1,h} \|\chi-u_h\|_{1,h} \nonumber\\
& {}\quad  + \int_{\Gamma_2} g(\lambda-\lambda_h)(\chi-u)\,ds\nonumber\\
& \le \frac{1}{2}\|e\|^2_{1,h} +\frac{1}{2}\left(\|u_h-z\|_{1,h}+\|\chi-u_h\|_{1,h}\right)^2
   + \|u_h-z\|_{1,h} \|\chi-u_h\|_{1,h} \nonumber\\
& {}\quad  + \int_{\Gamma_2}g(\lambda-\lambda_h)(\chi-u)\,ds.
\label{equ2}
\end{align}
Note that by \eqref{lag_1} and \eqref{lag_2}, we have
\begin{align*}
    \int_{\Gamma_2} g(\lambda- \lambda_h) (u_h-u)\, ds & =\int_{\Gamma_2} g \ \lambda \ u_h\,
    ds - \int_{\Gamma_2} g \ \lambda \ u\,ds -\int_{\Gamma_2} g \ \lambda_h \
    u_h\,ds + \int_{\Gamma_2} g \ \lambda_h \ u\,ds\\
    & \leq \int_{\Gamma_2} g \ |u_h|\,
    ds - \int_{\Gamma_2} g \ |u|\,ds -\int_{\Gamma_2} g \ |u_h|\,ds + \int_{\Gamma_2} g \
    |u|\,ds = 0.
\end{align*}
and
\begin{align*}
    \int_{\Gamma_2} g(\lambda- \lambda_h) (\chi-u_h)\, ds
    & \leq |\lambda- \lambda_h|_{*,h} \|\chi-u_h\|_{1,h}\\
    & \leq \epsilon |\lambda- \lambda_h|_{*,h}^2 + \frac{1}{4\epsilon}\|\chi-u_h\|_{1,h}^2.
\end{align*}
Hence,
\begin{align}
\|e\|^2_{1,h} & \lesssim \|u_h-z\|^2_{1,h} + \|\chi-u_h\|^2_{1,h}
+ \epsilon |\lambda- \lambda_h|_{*,h}^2. \label{err1}
\end{align}
Recalling (\ref{dual_norm}), we have
\[|\lambda-\lambda_h|_{*,h} = \|u-z\|_{1,h} \leq \|e\|_{1,h} + \|u_h-z\|_{1,h}.\]
Then, we obtain the following result
\begin{align*}
\|e\|_{1,h} + |\lambda-\lambda_h|_{*,h} & \lesssim \|u_h-z\|_{1,h} +
\|\chi-u_h\|_{1,h}.
\end{align*}
Using \eqref{conform} to bound $\|\chi-u_h\|_{1,h}$, the above inequality can be rewritten as
\begin{align}
\|e\|_{1,h} + |\lambda-\lambda_h|_{*,h} & \lesssim \|u_h-z\|_{1,h} +
\left(\sum_{e\in{\cal E}_h^0} h_e^{-1}\|[u_h]\|_e^2\right)^{1/2}.\label{bridge}
\end{align}

The relation \eqref{bridge} serves as a starting point for derivation of reliable and efficient 
error estimators of DG methods for a frictional contact problem. In this paper, we focus on 
the derivation and analysis of residual type error estimators derived from the 
inequality \eqref{bridge}. A similar approach can also be applied to recovery type error estimators.

\section{Reliable residual-type estimators}
\setcounter{equation}0

Now we follow the ideas in \cite{wang12} to obtain a posteriori
error estimators of DG methods for solving the frictional contact
problem. The detailed derivation and analysis of a posteriori
error estimators is given for the LDG method \cite{cockburn98}.
For other DG methods discussed in \cite{wang10}, similar 
results could be obtained by similar arguments.

To bound the first term $\|u_h-z\|_{1,h}$, we recall one result in
\cite{carstensen09}. Note that the a posteriori error analysis in \cite{carstensen09}
was only for the Poisson problem with homogenous Dirichlet boundary condition,
but it is easy to extend the result to general elliptic problems with Neumann
boundary conditions. For the second-order elliptic problem
\[-\Delta u + u = f \ {\rm in}\ \Omega, \quad u=0 \ {\rm on}\ \Gamma_1, 
\quad \frac{\partial u}{\partial n}=g \ {\rm on}\ \Gamma_2,\]
rewrite it as the first order system
\begin{equation}\label{Pois_first}
p = \nabla u, \ -\nabla \cdot p + u = f \ {\rm in} \ \Omega, \quad
u=0 \ {\rm on}\ \Gamma_1, \quad \frac{\partial u}{\partial n}=g \
{\rm on}\ \Gamma_2.
\end{equation}
Then the DG formulation for this problem is
\begin{align}
&\int_\Omega p_h \cdot \tau_h dx = - \int_\Omega u_h\,\nabla_h\cdot\tau_h dx 
+ \sum_{K\in\mathcal{T}_h}\int_{\partial K} \hat{u_h}\,n_K\cdot\tau_h ds
&\forall \,\tau_h \in W_h,\label{dg_form1}\\
&\int_\Omega (p_h \cdot \nabla_h v_h + u_h v_h)\, dx = \int_\Omega f\,v_h dx 
+ \sum_{K\in \mathcal{T}_h}\int_{\partial K}\hat{p_h}\cdot n_K v_h ds &\forall\,v_h\in V_h,
\label{dg_form2}
\end{align}
where $\hat{u_h}$ and $\hat{p_h}$ are numerical fluxes. Different
choices of the numerical fluxes lead to different DG methods. The
following theorem (see \cite{carstensen09}) holds for the LDG
method and other methods discussed in \cite{arnold02}.

\begin{theorem}\label{carstensen09}
Assume $u\in H^1_{\Gamma_1}(\Omega)$ and $p\in W:=[L^2(\Omega)]^2$
are the solution of the problem  $(\ref{Pois_first})$, and $u_h\in
V_h$ and $p_h\in W_h$ are the solution of the problem
$(\ref{dg_form1})$--$(\ref{dg_form2})$. Then,
\[\|p-p_h\| \leq C \left(\eta_* + \zeta_*\right),\]
where
\begin{align*}
\eta^2_* & := \sum_{K\in \mathcal{T}_h} h_K^2 \|{\rm div} p_h -u_h +f\|^2_K
+ \sum_{e\in {\cal E}_h^i} h_e \|[p_h]\|^2_e + \sum_{e\in {\cal E}_h^2} h_e \|p_h\cdot n-g\|^2_e, \\
\zeta^2_* &:= \sum_{e\in {\cal E}_h^0} h_e^{-1} \|[u_h]\|^2_e
\end{align*}
and $C$ is a mesh-size independent constant which depends only on
the domain $\Omega$ and the minimal angle condition.
\end{theorem}

From the relation between $p_h$ and $u_h$
(\cite{arnold02,carstensen09}), we deduce the following result.

\begin{corollary}\label{post_pois}
With the same notation as in Theorem \ref{carstensen09}, we have
\[\|\nabla u-\nabla_h u_h\| \leq C (\eta + \zeta_*),\]
where
\begin{align*}
\eta^2 & := \sum_{K\in \mathcal{T}_h} h_K^2 \|\Delta u_h - u_h +
f\|^2_K + \sum_{e\in {\cal E}_h^i} h_e \|[\nabla_h u_h]\|^2_e +
\sum_{e\in {\cal E}_h^2} h_e \|\nabla_h u_h\cdot n-g\|^2_e.
\end{align*}
\end{corollary}
\begin{proof}
By \cite[Lemma 7.2]{schotzau03},
\begin{align*}
\|r_0([v_h])\|^2 &\leq C \sum_{e\in {\cal E}_h^0} h_e^{-1}
\|[v_h]\|^2_e,\quad \quad \|l(\beta\cdot [v_h])\|^2 \leq C
\sum_{e\in {\cal E}_h^i} h_e^{-1} \|[v_h]\|^2_e, \quad\quad
\forall\, v_h\in V_h.
\end{align*}
From \cite[(3.9)]{arnold02}, we know that
\begin{align*}
p_h = \nabla_h u_h - r_0([\hat{u_h}-u_h]) - l(\{\hat{u_h}-u_h\}).
\end{align*}
Then
\begin{align*}
\|\nabla u-\nabla_h u_h\| &\leq \|\nabla u-p_h\| + \|p_h -\nabla_h
u_h\|\\ & \leq C \left(\eta_* + \zeta_*\right) +
\|r_0([\hat{u_h}-u_h])\| + \|l(\{\hat{u_h}-u_h\})\|.
\end{align*}
From the choices of numerical fluxes $\hat{u_h}$ in Table 3.1 of
\cite{arnold02}, we have
\begin{align*}
[\hat{u_h}-u_h] & = -[u_h]\ {\rm or}\ 0, \quad\quad
\{\hat{u_h}-u_h\} = - \beta\cdot [u_h] \ {\rm or}\ 0.
\end{align*}
So
\begin{align*}
\|r_0([\hat{u_h}-u_h])\| \leq C \sum_{e\in {\cal E}_h^0} h_e^{-1}
\|[u_h]\|^2_e,\quad \quad \|l(\{\hat{u_h}-u_h\})\| \leq C
\sum_{e\in {\cal E}_h^i} h_e^{-1} \|[u_h]\|^2_e,
\end{align*}
which implies
\begin{align*}
\|p_h -\nabla_h u_h\|\leq \zeta_* \quad {\rm and} \quad \|\nabla
u-\nabla_h u_h\| & \leq C \left(\eta_* + \zeta_*\right).
\end{align*}
Finally, by the inverse inequality and trace inequality, we get
\begin{align*}
\eta^2_* & = \sum_{K\in \mathcal{T}_h} h_K^2 \|{\rm div} p_h - u_h
+f\|^2_K + \sum_{e\in {\cal E}_h^i} h_e \|[p_h]\|^2_e + \sum_{e\in {\cal E}_h^2} h_e \|p_h\cdot n-g\|^2_e\\
& \leq 2\left( \eta^2 + \sum_{K\in \mathcal{T}_h} h_K^2 \|{\rm
div} (p_h - \nabla u_h) \|^2_K + \sum_{e\in {\cal
E}_h^i} h_e \|[p_h - \nabla_h u_h]\|^2_e + \sum_{e\in {\cal E}_h^2} h_e \|(p_h-\nabla_h u_h) \cdot n\|^2_e\right)\\
& \leq 2\eta^2 + 2 \sum_{K\in \mathcal{T}_h} h_K^2 \|{\rm div}
(p_h - \nabla u_h) \|^2_K + C \left(\sum_{K\in \mathcal{T}_h}
\|p_h - \nabla u_h \|^2_K + \sum_{K\in \mathcal{T}_h} h_K^2 |p_h -
\nabla u_h |^2_{1,K} \right)\\
& \leq 2\eta^2 + C \sum_{K\in \mathcal{T}_h} \|p_h - \nabla u_h
\|^2_K = 2\eta^2 + C \|p_h - \nabla_h u_h \|^2 \leq 2\eta^2 + C
\zeta_*^2.
\end{align*}
Therefore, $\eta_* \leq C \left(\eta +\zeta_*\right)$ and the
result is proved.
\end{proof}

Define the interior residuals and edge-based jumps
\begin{align*}
R_K&:=\Delta u_h - u_h +f  \quad{\rm for\ each}\ K\in\mathcal{T}_h,\\
R_e&:=[\nabla_h u_h]  \quad{\rm for\ each}\ e \in {\cal E}_h^i,
\quad R_e:=\nabla_h u_h\cdot n +g\lambda_h  \quad{\rm for\ each}\ e
\in {\cal E}_h^2.
\end{align*}
Then the local estimators are
\begin{align}\label{local_esti}
\eta_K&:= \left(h_K^2 \|R_K\|^2_K + \frac{1}{2}\sum_{e\in \partial
K\cap {\cal E}_h^i}h_e\|R_e\|^2_e + \sum_{e\in \partial
K\cap {\cal E}_h^2}h_e\|R_e\|^2_e\right)^{1/2},\\
\label{local_edge}
\eta_{\partial K}&:=\Big(\frac{1}{2}\sum_{e\in \partial K\cap {\cal E}_h^i}
h_e^{-1}\|[u_h]\|^2_e + \sum_{e\in \partial K\cap {\cal E}_h^1}
h_e^{-1}\|[u_h]\|^2_e\Big)^{1/2}.
\end{align}
Applying Corollary \ref{post_pois} to $\|u_h-z\|_{1,h}$, we obtain
from (\ref{post_obst})
\begin{align}
\|e\|_{1,h} + |\lambda-\lambda_h|_{*,h} & \lesssim  \left(\sum_{K\in\mathcal{T}_h}\eta_K^2 +
\sum_{K\in\mathcal{T}_h}\eta_{\partial K}^2 \right)^{1/2}. \label{err2}
\end{align}

\begin{theorem}\label{theorem1}
Let $u\in H^2(\Omega)$ and $u_h$ solve $(\ref{vi})$ and
$(\ref{dvi})$ respectively. Then we have the bound \eqref{err2}.
\end{theorem}

\section{Efficiency of the estimators}
\setcounter{equation}0

Now we consider lower bounds of the estimators. We follow the
standard argument of lower bounds of residual error estimators for
elliptic problems, see \cite[pp.\ 28--31]{ainsworth00}. First, we
introduce the bubble functions. Let $K \in \mathcal{T}_h$, and let
$\lambda_1$, $\lambda_2$ and $\lambda_3$ be the barycentric
coordinates on $K$. Then the interior bubble function $\varphi_K$
is defined by
\[\varphi_K = 27\lambda_1\lambda_2\lambda_3\]
and the three edge bubble functions are given by
\[\tau_1=4\lambda_2\lambda_3, \quad \tau_2=4\lambda_1\lambda_3,
  \quad \tau_3=4\lambda_1\lambda_2.\]
We list properties of bubble functions stated in Theorems 2.2 and
2.3 of \cite{ainsworth00} in the form of a lemma.

\begin{lemma}\label{lemma_bub}
For each $K \in {\cal T}_h$, $e \subset \partial K$, let
$\varphi_K$ and $\tau_e$ be the corresponding interior and edge
bubble functions. Let $P(K) \subset H^1(K)$ and $P(e)\subset
H^1(e)$ be finite-dimensional spaces of functions defined on $K$
or $e$. Then there exists a constant $C$ independent of $h_K$ such that for all $v\in
P(K)$,
\begin{align}
C^{-1}\|v\|^2_{K} \leq \int_K \varphi_K v^2 \, dx &\leq
C\|v\|^2_{K},\label{bub1}\\
C^{-1}\|v\|_{K} \leq \|\varphi_K v\|_{K} + h_K |\varphi_K
v|_{1,K} &\leq C\|v\|_{K},\label{bub2}\\
C^{-1}\|v\|^2_{e} \leq \int_e \tau_e v^2 \, ds &\leq
C\|v\|^2_{e},\label{bub3}\\
h_K^{-1/2}\|\tau_e v\|_{K} + h_K^{1/2} |\tau_e v|_{1,K} &\leq
C\|v\|_{e}.\label{bub4}
\end{align}
\end{lemma}

Denote
\[ a_K(u,v)=\int_K (\nabla u\cdot\nabla v + uv)\,dx. \]
Then for $u,v\in H^1(\Omega)$,
\[ a(u,v)=\sum_{K\in{\cal T}_h} a_K(u,v).\]
For all $v\in H^1_{\Gamma_1}(\Omega)$, noting that $[v]=0$ and
$[u-z]=0$ on $e\in {\cal E}_h^0$, we have
\begin{align}
\sum_{K\in{\cal T}_h} a_K(e,v)&=\sum_{K\in{\cal T}_h}
a_K(u_h-z,v)+a(z-u,v)
  =\sum_{K\in{\cal T}_h} a_K(u_h-z,v)+B_h(z-u,v) \nonumber\\
& = \sum_{K\in \mathcal{T}_h} \int_K \big(\nabla
(u_h-z)\cdot\nabla v
+(u_h-z)v\big)\, dx+\int_{\Gamma_2} g(\lambda-\lambda_h)v\,ds \nonumber\\
& = \sum_{K\in \mathcal{T}_h} \int_K \big(-\Delta (u_h-z) + u_h -z
\big)v \, dx+
\sum_{K\in \mathcal{T}_h} \int_{\partial K} \nabla(u_h-z)\cdot n_K v \, ds\nonumber\\
&\quad +\int_{\Gamma_2} g(\lambda-\lambda_h)v\,ds \nonumber\\
& = \sum_{K\in \mathcal{T}_h} \int_K (-\Delta u_h + u_h -f)v \,
dx+ \sum_{e\in {\cal E}_h^i} \int_{e} [\nabla u_h]\cdot v \, ds\nonumber\\
&\quad + \sum_{e\in {\cal E}_h^2} \int_{e} (\nabla u_h \cdot n +
g\lambda_h)  v \, ds + \int_{{\cal E}_h^2}
g(\lambda-\lambda_h)v\,ds. \label{ridu}
\end{align}
For each $K \in \mathcal{T}_h$, $\varphi_K$ and $\tau_e$ are
respectively the interior and edge bubble functions on $K$ or
$e\in {\cal E}_h^i \cup {\cal E}_h^2$. $\bar{R}_K$ is an
approximation to the interior residual $R_K$ from a suitable
finite-dimensional subspace. In (\ref{ridu}), choose
$v=\bar{R}_K\varphi_K$ on element $K$.  We know $\varphi_K$
vanishes on the boundary of $K$ by its definition, so $v$ can be
extended to be zero on the rest of domain as a continuous
function. Therefore, we get
\[a_K(e,\bar{R}_K\varphi_K)=\int_K R_K\bar{R}_K\varphi_K\, dx.\] Then
\[\int_K \bar{R}_K^2 \varphi_K\,dx=\int_K \bar{R}_K (\bar{R}_K -R_K)\varphi_K\,dx
+ a_K(e,\bar{R}_K\varphi_K).\] Applying the Cauchy-Schwarz
inequality and Lemma \ref{lemma_bub}, we obtain
\begin{align*}
\int_K \bar{R}_K (\bar{R}_K -R_K)\varphi_K \, dx &\leq \|\bar{R}_K
\varphi_K \|_{K} \|\bar{R}_K -R_K \|_{K} \lesssim \|\bar{R}_K
\|_{K} \|\bar{R}_K -R_K \|_{K}, \\
a_K(e,\bar{R}_K \varphi_K) &\leq \|e\|_{1,K} \|\bar{R}_K \varphi_K
\|_{1,K} \lesssim h_K^{-1}\|e\|_{1,K} \|\bar{R}_K \|_{K}.
\end{align*}
Use Lemma \ref{lemma_bub} again,
\[\|\bar{R}_K \|_{K}^2 \lesssim \int_K \bar{R}_K^2 \varphi_K dx.\]
Combining the above relations, we obtain
\[\|\bar{R}_K \|_{K} \lesssim\|\bar{R}_K -R_K \|_{K} + h_K^{-1}\|e\|_{1,K}.\]
Finally, by the triangle inequality $\|R_K \|_{K} \leq
\|R_K-\bar{R}_K \|_{K} + \|\bar{R}_K \|_{K}$, we get
\[\|R_K \|_{K} \lesssim\|\bar{R}_K -R_K \|_{K} + h_K^{-1}\|e\|_{1,K}.\]
Now choose the finite-dimensional subspace from which the
$\bar{R}_K$ come as the function space spanned by the local nodal
basis $\phi_K^{(i)}$ with $i = 1, 2, 3$.  Then, $\|\bar{R}_K -R_K
\|_{K}$ reduces to $\|f - \overline{f} \|_{K}$ where we take
\begin{equation}\label{f_h}
\overline{f} = \sum_{i=1}^3 f^i \phi_K^{(i)} \quad\quad {\rm with} \quad
f^i = (f,\phi_K^{(i)})_K/(1,\phi_K^{(i)})_K .
\end{equation}

For $e\in {\cal E}_{h}^2$, we obtain
\begin{align*}
a_{\omega_e}(u_h-u,\overline{R}_e \tau_e)&= \int_{\omega_e} R_K\overline{R}_e \tau_e dx
+ \int_{e} R_e\overline{R}_e \tau_e ds + \int_{e} g(\lambda- \lambda_h)\overline{R}_e \tau_e ds
\end{align*}
and therefore
\begin{align*}
\int_{e} \overline{R}_e^2 \tau_e ds 
& = \int_{e} \overline{R}_e (\overline{R}_e-R_e) \tau_e  ds +
a_{\omega_e}(u_h-u,\overline{R}_e \tau_e)\\
&{}\quad - \int_{\omega_e} R_K\overline{R}_e \tau_e dx
- \int_{e} g(\lambda- \lambda_h)\overline{R}_e \tau_e ds.
\end{align*}
From Lemma \ref{lemma_bub}, we estimate the terms in above
relation as
\begin{align*}
C^{-1}\|\overline{R}_e\|_e^2 & \leq \int_{e} \overline{R}_e^2
\tau_e \, ds,\\
\int_{e}
 \overline{R}_e (\overline{R}_e-R_e) \tau_e \, ds & \leq \|\overline{R}_e
 \tau_e\|_e \|\overline{R}_e-R_e\|_e \leq C \|\overline{R}_e
\|_e \|\overline{R}_e-R_e\|_e,\\
a_{\omega_e}(u_h-u,\overline{R}_e \tau_e) & \leq
\|u_h-u\|_{1,\omega_e} \|\overline{R}_e \tau_e\|_{1,\omega_e} \leq
C h_e^{-1/2} \|u_h-u\|_{1,\omega_e} \|\overline{R}_e\|_{e},\\
\int_{\omega_e} R_K \overline{R}_e \tau_e \, dx & \leq
\|R_K\|_{\omega_e} \|\overline{R}_e \tau_e\|_{\omega_e} \leq
C h_e^{1/2} \|R_K\|_{\omega_e} \|\overline{R}_e\|_{e},\\
\int_{e} g(\lambda- \lambda_h) \overline{R}_e \tau_e \, ds & \leq
|\lambda- \lambda_h|_{*,e} \|\overline{R}_e \tau_e\|_{1,\omega_e}
\leq C h_e^{-1/2} |\lambda- \lambda_h|_{*,e}
\|\overline{R}_e\|_{e}.
\end{align*}
Hence, we obtain
\begin{align}
\|R_e \|_{e}&\leq \|\overline{R}_e \|_{e} + \|R_e-\overline{R}_e
\|_{e}\nonumber\\ &\leq C \big(h_e^{-1/2}\|u_h-u\|_{1,\omega_e} +
h_e^{-1/2}|\lambda-\lambda_h|_{*,e}
   +h_e^{1/2}\|R_K-\overline{R}_K\|_{\omega_e} + \|R_e-\overline{R}_e\|_e
  \big).\label{eff_Re}
\end{align}
For $e\in {\cal E}_h^i$, let $\overline{R}_e$ be an approximation
to the jump $R_e$ from a suitable finite-dimensional space and let
$v=\overline{R}_e \tau_e$ in \eqref{ridu}.  By a similar argument, we have
\[\|R_e \|_{e} \leq C \big(h_e^{-1/2}\|u_h-u\|_{1,\omega_e}
   +h_e^{1/2}\|R_K-\overline{R}_K\|_{\omega_e} + \|R_e-\overline{R}_e\|_e \big).\]

Note that $\Delta u_h + u_h$ in $K$ and $\partial u_h/
\partial n_e$ on $e$ are polynomials. Hence, the terms
$\|R_K-\overline{R}_K\|_{K}$ and $\|R_e-\overline{R}_e\|_e$ can be
replaced by $\|f-\overline{f}\|_{K}$ and
$\|\lambda_h-\overline{\lambda}_h\|_e$, with discontinuous
piecewise polynomial approximations 
$\overline{\lambda}_h$. Then we obtain the efficiency bound of the
local error indicator $\eta_{K}$.

\begin{theorem}\label{theorem3}
Let $u$ and $u_h$ be the solutions of $(\ref{vi})$
and $(\ref{dvi})$, respectively, and $\eta_K$ be the estimator
$(\ref{local_esti})$. Then
\begin{equation}\label{effic}
\eta_K \leq C \left(|u-u_h|_{\omega_K} +
\sum_{e\in {\cal E}(K)\cap {\cal E}_{2}}|\lambda-\lambda_h|_{*,e} + h_K \|f-f_h\|_{\omega_K}
+\sum_{e\in {\cal E}(K)\cap {\cal E}_{2}}
h_e\|\lambda_h-\overline{\lambda}_h\|_e^2\right),
\end{equation}
where the constant $C$ is dependent on the angle condition and
independent of $h_K$.
\end{theorem}

\end{document}